\documentclass[11pt,reqno,a4paper]{amsart}
\usepackage[margin=.7in]{geometry}
\usepackage[usenames]{color}
\usepackage{amsmath,pdfsync,verbatim,graphicx,epstopdf,enumerate}
\usepackage[abbrev,nobysame]{amsrefs}
\usepackage[colorlinks=true]{hyperref}
\usepackage{cancel}
\usepackage[framemethod=tikz]{mdframed}
\hypersetup{allcolors=blue}
\allowdisplaybreaks
\numberwithin{equation}{section}

\newcommand{\lb}{\left(}

\newcommand{\rb}{\right)}
\newcommand{\PD}{\partial}

\newcommand{\Beq}{\begin{equation}}
	\newcommand{\Eeq}{\end{equation}}
\newcommand{\beq}{\begin{equation*}}
	\newcommand{\eeq}{\end{equation*}}
\newcommand{\bal}{\begin{align}}
	\newcommand{\eal}{\end{align}}

\usepackage{mathtools}

\usepackage[notref,notcite]{}

\newcommand{\bp}{\begin{prob}}
	\newcommand{\ep}{\end{prob}}
\newcommand{\bpr}{\begin{proof}}
	\newcommand{\epr}{\end{proof}}
\renewcommand{\o}{\omega}

\newcommand{\bel}[1]{\begin{equation}\label{#1}}
	\newcommand{\ee}{\end{equation}}

\newtheorem{theorem}{Theorem}[section]
\newtheorem{corollary}[theorem]{Corollary}

\newtheorem{lemma}[theorem]{Lemma}

\newtheorem{conjecture}{Conjecture}

\theoremstyle{definition}

\newcommand{\Rn}{\mathbb{R}^n}

\newcommand{\D}{\mathrm{d}}
\newcommand{\Lc}{\mathcal{L}}
\newcommand{\Rb}{\mathbb{R}}

\newcommand{\A}{\alpha}

\newcommand{\ve}{\varepsilon}

\newcommand{\Bb}{\mathbb{B}}
\usepackage[usenames]{color}
\usepackage{amsmath,pdfsync,verbatim,graphicx,epstopdf,enumerate}

\newcommand{\wt}{\widetilde}

\newcommand{\Sb}{\mathbb{S}}
\newcommand{\Sn}{\mathbb{S}^{n-1}}

\renewcommand{\o}{\omega}

\title[Null space of the backprojection operator for spherical mean transform]{On the null space of the backprojection
operator and Rubin's conjecture for the spherical mean transform}
\author[Agrawal,~ Ambartsoumian, ~Krishnan,~ Singhal]{Divyansh Agrawal$^\ast$, Gaik Ambartsoumian$^\dagger$, Venkateswaran P. Krishnan$^\ast$ and Nisha Singhal$^\ast$}

\address {$^{\ast}$ Centre for Applicable Mathematics, Tata Institute of Fundamental Research, Bangalore, India
\newline
E-mail:{\tt\  agrawald@tifrbng.res.in, vkrishnan@tifrbng.res.in, nisha2020@tifrbng.res.in}
\newline
Orcid:{\tt\ 0009-0003-5125-0640, 0000-0002-3430-0920, 0009-0006-3005-1986}}
\address{$^\dagger$ Department of Mathematics, The University of Texas at Arlington, Texas, USA
\newline
E-mail:{\tt \ gambarts@uta.edu}
\newline
Orcid:{\tt\ 0000-0002-1462-9964}}

\date{\today}

\begin{document}
\begin{abstract}
The spherical mean transform associates to a function $f$ its integral averages over all spheres. We consider the spherical mean transform for functions supported in the unit ball $\Bb$ in $\Rb^n$ for odd $n$, with the centers of integration spheres restricted to the unit sphere $\Sb^{n-1}$. In this setup, Rubin employed  properties of Erd\'elyi-Kober fractional integrals and analytic continuation to re-derive the explicit inversion formulas proved earlier by Finch, Patch, and Rakesh using wave equation techniques. As part of his work, Rubin stated a conjecture relating spherical mean transform, its associated backprojection operator and the Riesz potential. Furthermore, he pointed to the necessity of a detailed analysis of injectivity of the backprojection operator as a crucial step toward the resolution of his conjecture.
This article addresses both questions posed by Rubin by providing a characterization of the null space of the backprojection operator, and disproving the conjecture through the construction of an explicit counterexample. Crucial to the proofs is the range characterization for the spherical mean transform in odd dimensions derived recently by the authors.

\end{abstract}


\maketitle

\section{Introduction and statements of main results}

The purpose of this article is to address two questions posed by Rubin in his paper \cite{R} for the spherical mean transform (SMT) in odd dimensions. The first is a conjecture that he suggested about a relation between SMT and a Riesz potential, and the second question (which is related to his conjecture) is about the injectivity of the associated backprojection operator. 
Our first result is that the backprojection operator is not injective; in fact, we give a complete characterization of the null space of the backprojection operator.  Our second result is that the conjecture stated in \cite{R} is not true. We show this through an explicit counterexample. 

As in Rubin's work, our results specifically deal with the case of odd dimensions. In the recent work \cite{agrawal2023simple}, we derived a range characterization for SMT in odd dimensions, which is simpler than what previously existed in the literature. 
Our range characterization plays a pivotal role in proving the results described above.

The spherical mean transform maps a function to its integral averages over spheres with centers in $\Rn$. A formal dimension count shows that SMT depends on $n+1$ variables, while the function itself depends on  $n$ variables. One can make the inversion of SMT a formally determined problem
by restricting the centers of the spheres of integration to a hypersurface.  Motivated by potential applications in tomography, a common approach is to consider this hypersurface to be $\Sb^{n-1}$, and functions supported in $\Bb$, the unit ball in $\Rb^n$. 

For $f \in C_c^\infty (\Bb)$, SMT is defined as 
\begin{equation*}
\mathcal{M} f (p, t) = \frac{1}{\o_n} \int\limits_{\Sn} f(p+t\theta) \, \D S(\theta),
\end{equation*}
where $\o_n$ denotes the surface area of $\Sn$ and $\D S$ denotes the surface measure. 
Note that due to the support restriction on $f$, $\mathcal{M} f (\cdot, t) = 0$ for $t \geq 2$. In fact, $\mathcal{M}f(\cdot, t)=0$ for $t$ close enough to $0$ and $2$. Thus, $\mathcal{M}: C_c^\infty(\Bb) \to C_c^\infty (\Sn \times (0,2))$.

The spherical mean transform arises naturally in several tomographic applications, including, thermoacoustic and photoacoustic tomography, radar and sonar, as well as ultrasound reflectivity imaging. For this reason, the study of SMT in the context of inverse problems attracted significant attention in the recent past, including, inversion formulas, range characterization and reconstruction algorithms; see for instance \cites{Ambartsoumian2018, Ambartsoumian-Zarrad-Lewis, ambartsoumian2015inversion, ambartsoumian2014exterior, ambartsoumian2007thermoacoustic, AER, aramyan2020recovering, Finch-Haltmeir-Rakesh_even-inversion, Finch-P-R, K, nguyen2009family, norton1980reconstruction, norton1981ultrasonic, R, xu2002time, AKK, Agranovsky-Kuchment-Quinto, Salman_Article,  AN, ref:AmbKuch-range, finch2006range, LVN, agrawal2023simple}. 

Next, we present the explicit inversion formula for the SMT in odd dimensions derived by Finch, Patch and Rakesh in \cite{Finch-P-R}, and (using a simpler argument) by Rubin in \cite{R}. 

\begin{theorem}\cite[Theorem~3]{Finch-P-R}, \cite[Theorem 3.2]{R}
    Let $n\geq 3$ be odd, $f$ be a smooth function supported in  $\overline\Bb$, and $\mathcal{M}f(p,t)$ be known for all $p \in \Sn$ and all $t \in (0,2)$. Then 
    \begin{equation} \label{inv}
        \begin{split}
            f(x) &= c(n) \Delta \int\limits_{\Sn} \lb D^{n-3} t^{n-2} \mathcal{M} f(\theta,t) \rb \Big|_{t=|x-\theta|} \, \D S(\theta), \quad x \in \Bb \\
            &= c(n) \Delta \Big{(} P \lb D^{n-3}t^{n-2} \mathcal{M}f \rb(x)\Big{)},
        \end{split}
    \end{equation}
    where $D = \frac{1}{t} \frac{\D}{\D t}$, and
    \begin{align*}
        (PF)(x) = \frac{1}{\o_n} \int\limits_{\Sn} F(\theta, |x-\theta|) \D S(\theta),\quad x\in\mathbb{B}.
    \end{align*}
\end{theorem}

The notation in the above statement are slightly different from those of \cites{Finch-P-R,R}. Furthermore, we will not pay much attention to the constant $c(n)$, as this does not enter into the analysis that follows. As already mentioned, Rubin in \cite{R} gave an alternate and much simpler proof than that of \cite{Finch-P-R}, using Erd\'elyi-Kober fractional integrals and analytic continuation. He then stated the following conjecture. More precisely, the necessity part was proved in \cite{R} and was required for his proof of the above inversion formula, while the sufficiency part was stated as a conjecture. 

\begin{conjecture}\cite{R}\label{C1}
    A function $g \in C_c^\infty(\Sn \times (0,2))$ belongs to the range of the operator $f \mapsto \mathcal{M}f $ iff $P\lb D^{n-3} t^{n-2} g \rb$ belongs to the range $I^2[C_c^\infty(\Bb)]$, where $I^2$ is the Riesz potential defined by 
    \[
    (I^2f)(x) = \frac{\Gamma\lb \frac{n-2}{2}\rb}{4 \pi^{n/2}} \int\limits_{\Bb} \frac{f(y)}{|x-y|^{n-2}} \, \D y.
    \]
\end{conjecture}

Rubin suggested that analyzing the injectivity of the backprojection operator $P$ may be a way to resolve his conjecture. In this paper, we characterize the null space of the backprojection operator $P$, and show that the sufficiency part of Conjecture \ref{C1} is not valid by constructing an explicit counterexample. 
Here are the main results of our paper.

\begin{theorem}\label{Kernel-backprojection} A function $g\in C_c^{\infty}(\Sb^{n-1}\times (0,2))$ belongs to $\operatorname{Ker}(P)$ if and only if $g\in \operatorname{Range}(D^{n-2}t^{n-2}\mathcal{M})$.
\end{theorem}
\begin{theorem}\label{Thm:RubinConj}
    There exists a function $g\in C_c^{\infty}(\Sb^{n-1}\times (0,2))$ such that $P(D^{n-3}t^{n-2}g)=0$, but $g\notin \operatorname{Range}(\mathcal{M})$. Consequently, the sufficiency part of Conjecture \ref{C1}  is not valid.
\end{theorem}

Both Theorems \ref{Kernel-backprojection} and \ref{Thm:RubinConj} are based on an equivalent representation of the range characterization results of \cite{agrawal2023simple} in an integral form. This is the content of the next theorem. Interestingly, our proof \textit{directly} shows the equivalence of the distinct formulas stated in \cite{agrawal2023simple} and \cite{finch2006range} (comment below \cite[Theorem 3]{finch2006range}) as range conditions, that is, without invoking the fact that they are range characterizations for SMT.

\begin{theorem}\label{RangeCha-General}
    A function $g\in C_c^{\infty}(\Sb^{n-1}\times (0,2))$ is in $\operatorname{Range}(\mathcal{M})$ if and only if $P(D^{n-2} t^{n-2}g)=0$.
\end{theorem}

The proofs of these theorems are presented in the next section. We first treat the case of radial functions, and the general case follows as a consequence with suitable modifications. 

Throughout the paper we repeatedly use certain technical results, which are spelled out below for easy reference. The first one is the Fa\`a di Bruno formula, relating the higher order derivatives of the composition of two functions and the derivatives of the individual functions. This is a generalization of the usual chain rule to higher order derivatives (see, for instance, \cite{Krantz-Parks_primer}).

\begin{lemma}[Fa\`a di Bruno formula]
    Let $F$ and $G$ be two smooth functions of a real variable. The derivatives of the composite function $F \circ G$ in terms of the derivatives of $F$ and $G$ are expressed as
    \begin{equation*}
        \frac{\D^p}{\D t^p} F(G(t)) = \sum\limits_{q=1}^p F^{(q)}(G(t)) B_{p,q} (G^{(1)}(t), \dots, G^{(p-q+1)}(t)),
    \end{equation*}
    where $B_{p,q}$ are the Bell polynomials given by
    \[
    B_{p,q}(x_1, \dots, x_{p-q+1}) = \sum \frac{p!}{j_1! \dots j_{p-q+1}!} \lb \frac{x_1}{1!}\rb^{j_1}\cdots  \lb \frac{x_{p-q+1}}{(p-q+1)!}\rb^{j_{p-q+1}},
    \]
with the sum taken over all non-negative sequences, $j_1,\cdots, j_{p-q+1}$ such that the following two conditions are satisfied: 
\begin{align*}
   & j_1+ j_2+\cdots + j_{p-q+1}= q,\\
   &j_1+2j_2+ \cdots + (p-q+1) j_{p-q+1}=p.
\end{align*}
\end{lemma}

Since all main results of this paper use the differential operator $D$, it will be helpful to reformulate the above statement in terms of $D$. 
Multiplying the standard chain rule by $\frac{1}{t}$, the $D$-derivative of the composition of two functions can be re-written as
\[
D (F(G(t))) = F'(G(t)) DG(t),
\]
where $F'$ denotes the usual derivative of $F$. The following lemma is then an easy verification.
\begin{lemma}
    Let $F$ and $G$ be smooth functions of a real variable. Then the $D$-derivatives of the composite function $F \circ G$ satisfy the relation
    \begin{equation*}
        D^p F(G(t)) = \sum\limits_{q=1}^p F^{(q)}(G(t)) B_{p,q} ((DG)(t), \cdots, D^{(p-q+1)}G(t)).
    \end{equation*}
\end{lemma}

This paper requires only the following special case of the above formula proved in \cite{agrawal2023simple}.

\begin{lemma}[Fa\`a di Bruno  formula - special case]\label{FdB}
    Let $F$ and $G$ be smooth functions of a real variable such that $D^jG = 0$ for $j \geq 3$. Then the following identity holds
    \begin{align*}
        D^p F(G(t)) &= \sum\limits_{q \geq p/2}^p  \frac{p!}{(2q-p)! (p-q)! 2^{p-q}} F^{(q)}(G(t)) \lb DG(t)\rb^{2q-p} \lb D^2G(t) \rb^{p-q}.
    \end{align*}
\end{lemma}

Two other results employed in the article include the integration by parts formula involving the operator $D$ (proved by direct verification) and the Funk-Hecke theorem (e.g. see \cite{Natterer_book,Rubin-book}).

\begin{lemma}\label{lemma-IBP}
    For smooth functions $F$ and $G$, the following identity holds:
    \begin{align}
        \int\limits_{a}^b \PD_t D^k F \cdot G \, \D t &= \left [ \sum\limits_{l=0}^{k-1} (-1)^l D^{k-l} F \cdot D^{l} G \right ]_{t=a}^{b} + (-1)^k \int\limits_a^b \PD_t F \cdot D^{k} G \, \D t, 
    \end{align}
    where the sum is interpreted as empty for $k=0$.
\end{lemma}

\begin{theorem}
[Funk-Hecke]
If $\int\limits_{-1}^{1} |F(t)|(1-t^{2})^{\frac{n-3}{2}} \D t <\infty$, then for each $\eta\in \Sb^{n-1}$, 
\[
\int\limits_{\Sb^{n-1}}F\lb \langle \sigma,\eta\rangle \rb Y_{ml}(\sigma) \D S(\sigma) = \frac{\left|\Sb^{n-2}\right|}{C_{m}^{\frac{n-2}{2}}(1)}\lb\int\limits_{-1}^{1} F(t) C_{m}^{\frac{n-2}{2}}(t) (1-t^{2})^{\frac{n-3}{2}} \D t\rb Y_{ml}(\eta),
\]
where $\lvert \Sb^{n-2}\rvert$ denotes the surface measure of the unit sphere in $\Rb^{n-1}$, $C_{m}^{\frac{n-2}{2}}(t)$ are the Gegenbauer polynomials, and $Y_{ml}$ for $ 0\leq m<\infty$, 
$1\leq l\leq d_{m}=\frac{(2m+n-2)(n+m-3)!}{m!(n-2)!},d_0 =1$, are the spherical harmonics. 
\end{theorem} 
Finally, we recall the range characterization for SMT in odd dimensions proved in \cite{agrawal2023simple}. It is used in the proof of Theorem \ref{RangeCha-General} presented in the next section.  
\begin{theorem}[\cite{agrawal2023simple}]\label{RangeCharacterization-General}
    Let $\Bb$ denote the unit ball in $\mathbb{R}^n$ for an odd $n \geq 3$, and $k: = (n-3)/2$. A function $g \in C_c^\infty (\Sb^{n-1}\times (0,2))$ is representable as $g = \mathcal{M} f$ for $f\in C_c^\infty(\Bb)$ if and only if  for each $(m,l), m\geq 0, 0\leq l\leq d_m$, $h_{ml}(t)=t^{n-2}g_{ml}(t)$ satisfies the following two conditions:
    \begin{itemize}
        \item there is a function $\phi_{ml}\in C_c^{\infty}((0,2))$ such that 
        \Beq\label{GRC0}
            h_{ml}(t)= D^{m} \phi_{ml}(t),
        \Eeq
        \item the function $\phi_{ml}(t)$ satisfies 
        \begin{equation}\label{GRC}
           [\Lc_{m+k}\phi_{ml}](1-t)= [\Lc_{m+k}\phi_{ml}](1+t),
        \end{equation}
        with
        \[
         \Lc_{m+k} = \sum\limits_{p=0}^{m+k} \frac{ (m+k+p)! }{(m+k-p)! p! 2^p} (1-\cdot)^{m+k-p} D^{m+k-p} \mbox{ and }  D = \frac{1}{t} \frac{\D}{\D t}.
        \]
    \end{itemize}
\end{theorem}

\section{Proofs} 
We first prove the radial version of Theorem \ref{RangeCha-General} and then move on to the 
general case. Henceforth, the letter $C$ is used to denote dimensional constants, whose value may change from line to line in a given computation. The exact value of the constants can be computed, but does not affect the analysis.
\begin{theorem}\label{Thm:EquivRange}
    Let $g\in C_c^{\infty}((0,2))$. Then $g\in \operatorname{Range}(\mathcal{M})$ if and only if $P(D^{n-2} t^{n-2}g)=0$.
\end{theorem}
\bpr 
Define $h(t)=t^{n-2}g(t)$ and let $k=\frac{n-3}{2}$. For every $x\in\mathbb{B}$, one can compute
\Beq
\begin{aligned}
P(D^{n-2} h)(x)&=\int\limits_{\Sb^{n-1}}[D^{n-2}h](|x-\theta|) \D S(\theta)\\
&=\int\limits_{\Sb^{n-1}} [D^{n-2}h]\lb \sqrt{|x|^2+1-2x\cdot \theta}\rb \D S(\theta)\\
&=C\int\limits_{-1}^{1} [D^{n-2}h]\lb \sqrt{|x|^2+1-2|x|t}\rb \lb 1-t^{2}\rb^{k} \D t,
\end{aligned} 
\Eeq
using the Funk-Hecke formula in the last step.
Making the substitution $u=\sqrt{|x|^2+1-2|x|t}$ yields 
\[
P(D^{n-2}h)(x)=\frac{C}{|x|^{n-2}}\int\limits_{1-|x|}^{1+|x|}u[D^{2k+1}h](u)\lb 4|x|^2-(1+|x|^2-u^2)^2\rb^{k} \D u.
\]

Let us denote 
\[
A(x,u)=1+|x|^2-u^2,
\]
and 
\[
B(x,u)=4|x|^2-A^2(x,u).
\]
Observe that
\begin{align*}
    A(x, 1\pm |x|) &= \mp 2|x|, \\
    B(x, 1 \pm |x|) &= 0.
\end{align*}
Due to the above equality, one can apply integration by parts $k$ times without picking up any boundary terms and obtain  
\Beq\label{Eq:RangeCh-Radial}
P(D^{n-2}h)(x)=\frac{C}{|x|^{n-2}}\int\limits_{1-|x|}^{1+|x|}u[D^{k+1}h](u) D^{k}B^{k} \D u.
\Eeq
We will transfer another $k$ derivatives using integration by parts, but this time picking up boundary terms. Before performing that computation, let us find an expression for 
\[
D^{k+l} B^k \quad \mbox{for } \quad 0 \leq l \leq k.
\]
Observe that 
\begin{align*}
    D B &= 4 A, \\
     D^2 B &= -8, \\
    \mbox{and} \quad D^j B &= 0 \quad \mbox{for } j \geq 3.
\end{align*}
Invoking the special case of Fa\`a di Bruno formula (see Lemma \ref{FdB}), one gets 
\begin{align*}
    D^{k+l} B^k &= \sum\limits_{i \geq \frac{k+l}{2}}^k (-1)^{k+l-i} \frac{k!(k+l)! 2^{2i}}{(k-i)! (2i-k-l)! (k+l-i)!} A^{2i-k-l} B^{k-i}.
\end{align*}
Substituting this above leads to 
\begin{align*}
    P(D^{n-2}h)(x) &= \frac{C}{|x|^{2k+1}} \left [ \sum\limits_{l=0}^{k-1} (-1)^l D^{k-l} h  \sum\limits_{i \geq \frac{k+l}{2}}^k  \frac{(-1)^{k+l-i} k!(k+l)! 2^{2i}}{(k-i)! (2i-k-l)! (k+l-i)!} A^{2i-k-l} B^{k-i} \right ]_{1-|x|}^{1+|x|} \\
    & \qquad +  (-1)^k \frac{C}{|x|^{2k+1}} \int\limits_{1-|x|}^{1+|x|} \PD_t h \lb \frac{(-1)^k 2^{2k}k! (2k)! }{k!}\rb \, \D t. 
\end{align*}
Since $B(x, 1 \pm |x|) = 0$, only the $i=k$ term survives in the first expression on the right to give
\begin{align*}
    P(D^{n-2}h)(x) &= \frac{C}{|x|^{2k+1}} \left [ \sum\limits_{l=0}^{k-1} \frac{k! (k+l)! 2^{2k}}{(k-l)! l!} A^{k-l} D^{k-l} h  \right]_{1-|x|}^{1+|x|} \\
    &\qquad + \frac{C}{|x|^{2k+1}} 2^{2k}(2k)!  [h]_{1-|x|}^{1+|x|}.
\end{align*}
Writing it out, we have
\Beq
\begin{aligned}\label{RC-simplified}
    P(D^{n-2}h)(x) &= \frac{C}{ |x|^{2k+1}} \Bigg [ \sum\limits_{l=0}^{k} \frac{(-1)^{k-l} 2^{k-l} k! (k+l)!}{(k-l)! l!} |x|^{k-l} \left [D^{k-l} h \right ](1+|x|) \\
    & \quad \hspace{2cm} \quad - \sum\limits_{l=0}^{k} \frac{ 2^{k-l} k! (k+l)!}{(k-l)! l!} |x|^{k-l} \left [D^{k-l} h \right ](1-|x|) \Bigg ].
    \end{aligned}
    \Eeq
    In other words,
    \Beq
    P(D^{n-2}h)(x) = \frac{C}{|x|^{2k+1}} \lb [\Lc_k h](1+|x|) - [\Lc_k h](1-|x|) \rb,
\Eeq
where $\Lc_k$ is the linear differential operator of order $k$, defined by
\[
\Lc_k = \sum\limits_{l=0}^k \frac{ (k+l)! }{(k-l)! l! 2^l} (1-\cdot)^{k-l} D^{k-l}.
\]
Using Theorem \ref{RangeCharacterization-General}, we observe that $g\in \operatorname{Range}(\mathcal{M})$ if and only if $P(D^{n-2}t^{n-2}g)=0$. This concludes the proof of radial case.
\epr

Next we will prove the general case, that is, Theorem \ref{RangeCha-General}. We will need the following three lemmas. 

\begin{lemma}\label{FTC}
    Let $U \in C_c^\infty((0,2))$. Then 
    \[
    \int\limits_0^2 s^{2j+1} U(s) \, \D s =0  \mbox{ for all } 0 \leq j \leq n
    \]
    if and only if $U =  D^{n+1} V$ for some $V \in C_c^\infty((0,2))$.
\end{lemma}
\bpr
We use induction on $n$. Assume $n=0$. The ``if" part is obvious. Let us consider the ``only if" part. Since 
    \[
    \int\limits_0^2 s\, U(s) \, \D s = 0,
    \]
    the function $V(t) = \int\limits_{0}^t s\, U(s) \, \D s$ is compactly supported, smooth, and satisfies $DV=U$. 

    Next, let us assume the result is true for all $0\leq j\leq n$. As in the base case, the ``if'' part of the induction step is straightforward, so we prove the ``only if'' part here. Suppose
    \[
    \int\limits_0^2 s^{2j+1}U(s)=0 \mbox{ for all } 0\leq j\leq n+1.
    \]
    From the induction assumption, it follows that 
    \[
    U=D^{n+1} V \mbox{  for some } V\in C_c^{\infty}((0,2)).
    \]
         Combining the last two relations with $j=n+1$ leads to
    \[
    \int\limits_0^2 s^{2n+3} D^{n+1} V(s) \D s =0.
    \]
    Applying repeated integration by parts yields
    \[
    \int\limits_0^2 s V(s)\D s=0,
    \]
    which implies that $V=D\wt{V}$ for some $\wt{V}\in C_c^{\infty}((0,2))$, i.e.
    $U=D^{n+2} \wt{V} \mbox{ for } \wt{V}\in C_c^{\infty}((0,2)).$
\epr
\begin{lemma} Let $B(r,u)=4r^2-(1+r^2-u^2)^2$.
    Then for any $m \geq 1$,  $B^m$ can be expressed as 
    \[
    B^m(r,u) = \sum\limits_{j=0}^{2m} q_{j,2m}(u^2) r^{2j},
    \]
    where $q_{j,2m}$ is a polynomial of degree exactly $2m-j$.
\end{lemma}
\bpr
By regrouping the terms of $B(r,u)$ and applying the multinomial theorem we get
\[
\begin{aligned} 
B^{m}(r,u)& =\sum\limits_{i+j+k=m} \frac{m!}{i! j! k !} (-1)^{i} r^{4i} 2^{j}r^{2j}(1+u^2)^{j}(-1)^{k}(1-u^2)^{2k}\\
&=\sum\limits_{i+j+k=m} \frac{2^{j}m!}{i! j! k !} (-1)^{i+k} r^{4i+2j} (1+u^2)^{j}(1-u^2)^{2k}.
\end{aligned} 
\] 
One can rearrange the terms of the above sum in the increasing order of powers of $r^2$ to obtain
\[
\begin{aligned} 
B^{m}(r,u) &=(-1)^m\sum\limits_{\A=0}^{2m}(-2)^{\A} r^{2\A} \sum\limits_{i=0}^{[\A/2]} \frac{m!}{2^{2i} i! (\A-2i)!(m+i-\A)!}(1+u^2)^{\A-2i}(1-u^2)^{2m+2i-2\A}\\
&=(-1)^m\sum\limits_{\A=0}^{2m}(-2)^{\A} r^{2\A} \sum\limits_{i=0}^{[\A/2]} \frac{1}{4^i}{m\choose \A-2i}{m-\A+2i\choose m+i-\A} (1+u^2)^{\A-2i}(1-u^2)^{2m+2i-2\A} \\
&=\sum\limits_{\A=0}^{2m} q_{\A,2m}(u^2) r^{2\A},
\end{aligned} 
\]
where
\[
q_{\A,2m}(u^2) = (-1)^m (-2)^{\A} \sum\limits_{i=0}^{[\A/2]} \frac{1}{4^i} \binom{m}{\A-2i} \binom{m-\A+2i}{m+i-\A} (1+u^2)^{\A-2i} (1-u^2)^{2m+2i-2\A}.
\]
Observe that $q_{\A,2m}$ is a polynomial of degree $2m-\A$, where the coefficient of the highest degree term is $(-1)^m (-2)^{\A} \sum\limits_{i=0}^{[\A/2]} \frac{1}{4^i}{m\choose \A-2i}{m-\A+2i\choose m+i-\A}$, which is non-zero. This finishes the proof. \epr

\begin{lemma} \label{FTC2}
    Let $\ve>0$ and $h$ be a function such that 
    \[
    \int\limits_\epsilon^{2-\epsilon} u h(u) B^k(r,u) \, \D u = 0, \mbox{ for some }k \geq 1 \mbox{ and all } r \in (1-\epsilon, 1).
    \]
    Then,
    \[
    \int\limits_{\epsilon}^{2-\epsilon} u^{2j+1} h(u) \, \D u = 0, \quad \text{for all} \quad 0 \leq j \leq 2k.
    \]
\end{lemma}
\begin{proof}
    Note that the given quantity vanishes as a polynomial in $r$, hence each of its coefficients must also vanish. Therefore, using the previous lemma,
    \[
    \int\limits_{\epsilon}^{2-\epsilon} q_{j,2k} (u^2) u h(u) \D u = 0,  \quad \text{for each} ~~ 0 \leq j \leq 2k.
    \]
    Since $q_{j,2k}$ is a polynomial of degree \emph{exactly} $2k-j$, (starting with $j=2k$ and working backwards) we have 
    \[
    \int\limits_{\epsilon}^{2-\epsilon} \lb u^2\rb^{2k-j} u h(u) ~\D u = 0, \mbox{ for each } 0 \leq j \leq 2k.
    \]
    The proof is complete.
\end{proof}
\bpr[Proof of Theorem \ref{RangeCha-General}]
Let us expand $g$ in spherical harmonics: 
\[
g(\theta,t)=\sum\limits_{m=0}^{\infty}\sum\limits_{l=1}^{d_{m}} g_{ml}(t) Y_{ml}(\theta),
\]
with $g_{ml}\in C_c^{\infty}((0,2))$.
Define $h(\theta,t)=t^{n-2}g(\theta, t)$, and 
$h_{ml}(t)=t^{n-2}g_{ml}(t)$. Then
\[
[D^{n-2}h](\theta,t)=\sum\limits_{m=0}^{\infty}\sum\limits_{l=1}^{d_{m}} D^{n-2} h_{ml}(t) Y_{ml}(\theta),
\]
where $D=\frac{1}{t}\frac{\D}{\D t}$. Applying the backprojection operator to the above expression, one gets
\[
\begin{aligned}
P(D^{n-2} h)(x)&=\sum\limits_{m=0}^{\infty}\sum\limits_{l=1}^{d_m}\int\limits_{\Sb^{n-1}}[D^{n-2}h_{ml}](|x-\theta|) Y_{ml}(\theta) \D S(\theta)\\
&=\sum\limits_{m=0}^{\infty}\sum\limits_{l=1}^{d_{m}}\int\limits_{\Sb^{n-1}} [D^{n-2}h_{ml}]\lb \sqrt{|x|^2+1-2x\cdot \theta}\rb Y_{ml}(\theta)\D S(\theta)\\
&=C \sum\limits_{m=0}^\infty \sum\limits_{l=1}^{d_m} \Bigg{\{} \int\limits_{-1}^1 [D^{n-2} h_{ml}] \lb \sqrt{1+|x|^2-2|x|t} \rb C_m^{\frac{n-2}{2}} (t) \lb 1-t^2 \rb^{(n-3)/2} \D t \Bigg{\}} Y_{ml}\lb\frac{x}{|x|}\rb,
\end{aligned} 
\]
where in the last step, we used the Funk-Hecke theorem. Also recall that $C_m^{\A}(t)$ are the Gegenbauer polynomials defined by
\[
C_m^{\A}(t)=K_m (1-t^2)^{-\A+\frac{1}{2}}\frac{\D^{m}}{\D t^{m}}\lb 1-t^{2}\rb^{m+\A-\frac{1}{2}},
    \]
    where 
    \[
    K_m=\frac{(-1)^{m} \Gamma(\A + \frac{1}{2})\Gamma(m+2\A)}{2^m m!\Gamma(2\A)\Gamma(m+\A+\frac{1}{2})}.
    \]
    Substituting this above, and letting $|x|=r$, we get, 
    \[
    P(D^{n-2}h)(x)=\frac{C}{r}\sum\limits_{m=0}^\infty K_m\sum\limits_{l=1}^{d_m} \Bigg{\{} \int\limits_{1-r}^{1+r}u[D^{n-2}h_{ml}](u)\Big{\{}\frac{\D^{m}}{\D t^{m}}\lb 1-t^{2}\rb^{m+\frac{n-3}{2}}\Big{\}}\Big{|}_{t=\frac{1+r^2-u^2}{2r}}\D u\Bigg{\}}Y_{ml}\lb\frac{x}{r}\rb.
    \]
    A simple application of Fa\'a di Bruno formula gives  
    \[
    (-r)^{m}D_u^{m}\lb 1-\lb \frac{1+r^2-u^2}{2r}\rb^2\rb^{m+\frac{n-3}{2}}=\frac{\D^{m}}{\D t^{m}}\lb 1-t^{2}\rb^{m+\frac{n-3}{2}}\Big{\}}\Big{|}_{t=\frac{1+r^2-u^2}{2r}}.
    \]
    Then 
     \Beq\label{IntegralRange_General}
     \begin{aligned}
    P(D^{n-2}h)(x)&=C\sum\limits_{m=0}^\infty K_m\sum\limits_{l=1}^{d_m} r^{m-1}\Bigg{\{} \int\limits_{1-r}^{1+r}u[D^{n-2}h_{ml}](u)D_u^{m}\lb 1-\lb \frac{1+r^2-u^2}{2r}\rb^2\rb^{m+\frac{n-3}{2}}\D u\Bigg{\}}Y_{ml}\lb\frac{x}{r}\rb\\
    &=\frac{C}{r^{2k+1}}\sum\limits_{m=0}^\infty K_m\sum\limits_{l=1}^{d_m} \frac{1}{r^{m}}\Bigg{\{} \int\limits_{1-r}^{1+r}u[D^{n-2}h_{ml}](u)D_u^{m}\lb 4r^2-(1+r^2-u^2)^2\rb^{m+\frac{n-3}{2}}\D u\Bigg{\}}Y_{ml}\lb\frac{x}{r}\rb.
    \end{aligned} 
    \Eeq
    Now suppose $g\in \operatorname{Range}(\mathcal{M})$. Then by Theorem \ref{RangeCharacterization-General}, $h_{ml}=D^{m}\phi_{ml}$ for $\phi_{ml}\in C_c^{\infty}((0,2))$ and $\phi_{ml}$ satisfies \eqref{GRC}. Substituting $h_{ml}=D^{m}\phi_{ml}$ into \eqref{IntegralRange_General}, one gets 
    \[
    P(D^{n-2}h)(x)=\frac{C}{r^{2k+1}}\sum\limits_{m=0}^\infty K_m\sum\limits_{l=1}^{d_m}\frac{1}{r^{m}} \Bigg{\{} \int\limits_{1-r}^{1+r}u[D^{m+n-2}\phi_{ml}](u)D^{m}\lb 4r^2-(1+r^2-u^2)^2\rb^{m+\frac{n-3}{2}}\D u\Bigg{\}}Y_{ml}\lb\frac{x}{r}\rb.
    \]
    As before, we can transfer $D$ derivatives $k$ times in integration by parts without picking boundary terms. Then
    \[
     P(D^{n-2}h)(x)=\frac{C}{r^{2k+1}}\sum\limits_{m=0}^\infty K_m\sum\limits_{l=1}^{d_m}\frac{1}{r^{m}} \Bigg{\{} \int\limits_{1-r}^{1+r}u[D^{m+k+1}\phi_{ml}](u)D^{m+k}\lb 4r^2-(1+r^2-u^2)^2\rb^{m+k}\D u\Bigg{\}}Y_{ml}\lb\frac{x}{r}\rb.
    \]
    The expression on the right is exactly as in \eqref{Eq:RangeCh-Radial} with $k$ replaced by $m+k$. Hence if $\phi_{ml}$ satisfies \eqref{GRC} for each $m,l$, we have that $P(D^{n-2}h)(x)=0$.

    Let us now show the other implication. Suppose $P(D^{n-2}h)(x)=0$. Then proceeding as before, we have from \eqref{IntegralRange_General},
    \[
    \int\limits_{1-r}^{1+r}u[D^{2k+1}h_{ml}](u)D^{m}\lb 4r^2-(1+r^2-u^2)^2\rb^{m+k}\D u=0.
    \]
Let $\epsilon > 0$ be such that $\mbox{supp}(h_{ml})\subset [\ve, 2-\ve]$. Choosing $r$ to be $1-\ve'$ for any $\ve'<\ve$ and integrating by parts $m$ times, one gets
    \[
    \int\limits_{0}^{2}u[D^{m+2k+1}h_{ml}](u)\lb 4(1-\ve')^2-(1+(1-\ve')^2-u^2)^2\rb^{m+k}\D u=0.
    \]
    The expression on the left is a polynomial in $\ve'$. Since this polynomial vanishes for all $\ve'$ in an interval, using Lemma \ref{FTC2}, we obtain: 
\[
    \int\limits_0^2 u^{2j+1} D^{m+2k+1}h_{ml}(u)=0 \mbox{ for all } 0\leq j\leq 2m+2k.
    \]
    Therefore, as a consequence of Lemma \ref{FTC}, one can conclude that $D^{m+2k+1}h_{ml}= D^{2m+2k+1}\phi_{ml}$ for some $\phi_{ml}\in C_c^{\infty}((0,2))$. On the space of compactly supported functions, the operator $D$ has a zero kernel, and therefore  
    \[
    h_{ml}(u)=D^{m} \phi_{ml}.
    \]
    Substituting this in \eqref{IntegralRange_General}, and integrating by parts as before, one can see that $\phi_{ml}$ satisfies \eqref{GRC}. Thus, $g \in \mathrm{Range}(\mathcal{M})$ as a consequence of Theorem \ref{RangeCharacterization-General}.
\epr
Next we use Theorem \ref{RangeCha-General} to prove Theorem \ref{Kernel-backprojection}. 
As before, we start with functions $g\in C_c^{\infty}(\Sb^{n-1}\times (0,2))$ independent of the angular variable. 
\begin{theorem}\label{Thm:RC-Radial} A function $g\in C_c^{\infty}((0,2))$ is such that $g\in \operatorname{Ker}(P)$ if and only if  $g\in \operatorname{Range}(D^{n-2}t^{n-2}\mathcal{M})$.
\end{theorem}

\bpr
From Theorem \ref{Thm:EquivRange}, we know that a function $\wt{g}\in \operatorname{Range}(\mathcal{M})$ if and only if $P(D^{n-2}t^{n-2}\wt{g})=0$. Therefore, $g\in \operatorname{Range}(D^{n-2}t^{n-2}\mathcal{M})$ implies $Pg=0$. 

Let us show the reverse implication. Suppose $g\in C_c^{\infty}((0,2))$ such that $g\in \operatorname{Ker}(P)$. We would like to prove that there exists an $f\in C_c^{\infty}(\Bb)$ such that $g=D^{n-2}t^{n-2}\mathcal{M}f$.

Consider 
\begin{align*}
    Pg(x) &= \frac{1}{\o_n} \int\limits_{\Sn} g(|x-\theta|) \, \D S(\theta) \\
    &= \frac{1}{\o_n} \int\limits_{\Sn} g\lb\sqrt{1+|x|^2-2|x| \theta \cdot \frac{x}{|x|}}\rb \, \D S(\theta).
    \intertext{Applying the Funk-Hecke theorem, and writing all the constants as $C$, we have}
    Pg(x) &= C \int\limits_{-1}^1 g\lb\sqrt{1+|x|^2-2|x| t}\rb \, \lb1-t^2 \rb^{\frac{n-3}{2}} \, \D t.
    \intertext{The substitution $1+|x|^2-2|x|t = u^2$ yields}
    Pg(x) &= \frac{C}{2^{k}|x|^{2k+1}} \int\limits_{1-|x|}^{1+|x|} u g(u) \lb 4|x|^2 - \lb 1+|x|^2-u^2\rb^2 \rb^{k} \, \D u.
\end{align*}
Denoting $|x|=r$, we get
\Beq\label{Eq2.4}
0=Pg(x)=\frac{C}{2^{k}r^{2k+1}} \int\limits_{1-r}^{1+r} u g(u) \lb 4r^2 - \lb 1+r^2-u^2\rb^2 \rb^{k} \, \D u.
\Eeq
Next, let $\operatorname{supp}(g)\subset(\ve, 2-\ve)$.  Then for any $\ve'<\ve$, letting $r=1-\ve'$, and using the structure of $\operatorname{supp}(g)$: 
\[
0=\int\limits_{0}^{2} u g(u) \lb 4(1-\ve')^2 - \lb 1+(1-\ve')^2-u^2\rb^2 \rb^{k} \, \D u.
\]
We then have, using Lemma \ref{FTC2},
\[
\int\limits_0^{2} u^{2j+1} g(u)=0 \mbox{ for all } 0\leq j\leq 2k.
\]
Based on Lemma \ref{FTC}, one has
\[
g(u)=D^{2k+1} h(u) \mbox{ for } h\in C_c^{\infty}((0,2)).
\]
Since by assumption, $Pg(x)=0$, we get,
\[
P(D^{2k+1} h(t))=P\lb D^{2k+1} t^{n-2} \lb \frac{1}{t^{n-2}} h(t)\rb\rb=0.
\]
Note that due to the support condition on $h$, $\frac{1}{t^{n-2}} h(t)\in C_c^{\infty}((0,2))$. By Theorem \ref{Thm:EquivRange}, we have that $\frac{1}{t^{n-2}} h(t)\in \operatorname{Range}(\mathcal{M})$. In other words, there exists a radial $f\in C_c^{\infty}(\Bb)$ such that $h(t)=t^{n-2}\mathcal{M}f(t)$. Consequently, $g=D^{2k+1}t^{n-2}\mathcal{M}f(t)$, that is $g\in \operatorname{Range}(D^{n-2}t^{n-2}\mathcal{M})$. This concludes the proof.
\epr
The previous result can be generalized as follows. 
\begin{corollary}
    Let $g \in C_c^\infty((0,2))$. For each $0 \leq l \leq 2k+1$, $g \in \mathrm{Ker}(PD^{2k+1-l})$ if and only if $g \in \mathrm{Range}(D^l t^{n-2} \mathcal{M})$.
\end{corollary}

\noindent This follows as a straightforward consequence of the above proof. One only needs to observe that $D^{r}$, on the space of compactly supported smooth functions, has a trivial kernel. We skip the proof.

Let us now generalize the previous result to the following. The proof follows by minor modifications of Theorems \ref{Thm:RC-Radial} and \ref{RangeCha-General}. Note that the case $l=2k+1$ is precisely Theorem \ref{Kernel-backprojection}.
\begin{theorem}
    Let $g \in C_c^\infty(\Sn \times (0,2))$. For each $0 \leq l \leq 2k+1$, $g \in \mathrm{Ker}(PD^{2k+1-l})$ if and only if $g \in \mathrm{Range}(D^l t^{n-2} \mathcal{M})$.
\end{theorem}

Finally, we present a counterexample to Conjecture \ref{C1}. 

\bpr[Proof of Theorem \ref{Thm:RubinConj}]

We construct a non-trivial $g\in C_c^\infty((0,2))$ such that  $g \in \operatorname{Ker}(P D^{n-3} t^{n-2})$, but $g \notin \mathrm{Range}(\mathcal{M})$. This would give a counterexample to the conjecture since $0=P(D^{n-3}t^{n-2}g)\in I^2[C^{\infty}(\Bb)]$ (being the image of the $0$ function), but $g\notin \operatorname{Range}(\mathcal{M})$. 

Consider a non-trivial $\wt{g}\in C_c^{\infty}((0,2))$, such that $\wt{g}\in \operatorname{Range}(\mathcal{M})$. Then, it follows from Theorem \ref{Thm:EquivRange} that $h=t^{n-2}\wt{g}\in \operatorname{Ker}(PD^{n-2})$. Since $n\geq 3$, it implies that $Dh\in \operatorname{Ker}(PD^{n-3})$. A straightforward computation yields
\[
Dh(t)=(n-2)t^{n-4} \wt{g}(t)+ t^{n-3} \wt{g}\,'(t)=t^{n-2}\left[ \frac{(n-2)\wt{g}(t)}{t^2} + \frac{\wt{g}\,'(t)}{t}\right].
\]
Define 
\[
g(t)= \frac{(n-2)\wt{g}(t)}{t^2} + \frac{\wt{g}\,'(t)}{t}.
\]
Note that $g(t)\in C_c^{\infty}((0,2))$ and $g(t)$ is not the zero function. The latter claim follows from the fact that if $g(t)\equiv 0$, then $\wt{g}\,(t)$ solves a homogeneous linear ODE with a zero initial condition, 
implying that $\wt{g}\equiv 0$ and contradicting our choice of $\wt{g}$. We now have 
\[
P(D^{n-3}t^{n-2}g(t))=0.
\]
Suppose there exists an $f\in C_c^{\infty}(\Bb)$ such that $\mathcal{M}f=g$. Then the inversion formula from \cite{Finch-P-R} (see \eqref{inv}) gives  $f\equiv 0$, and hence $g\equiv 0$. This contradicts the fact that $g\not\equiv 0$. Therefore $g\notin \mbox{ Range}(\mathcal{M})$, thus showing that the sufficiency part of Rubin's conjecture is not valid.

In fact, the above argument can be generalized as follows. Consider $ \Tilde{g} = g + \mathcal{M} f$ for $g$ as above and $f \in C_c^\infty(\Bb)$. Then  $\Tilde{g} \in C_c^\infty(\Sn \times (0,2))$, and $P(D^{n-3} t^{n-2} \Tilde{g}) = P(D^{n-3} t^{n-2} \mathcal{M}f)$ belongs to the range $I^2(C_c^\infty(\Bb))$ (by necessity part of Conjecture~\ref{C1}), but $\Tilde{g} \notin \mathrm{Range}(\mathcal{M})$ (for otherwise $g \in \mathrm{Range}(\mathcal{M})$).
\epr

\bibliographystyle{plain}
\bibliography{references}

\begin{bibdiv}
\begin{biblist}

\bib{AKK}{incollection}{
      author={Agranovsky, Mark},
      author={Kuchment, Peter},
      author={Kunyansky, Leonid},
       title={On reconstruction formulas and algorithms for the thermoacoustic tomography},
        date={2017},
   booktitle={Photoacoustic {I}maging and {S}pectroscopy},
   publisher={CRC press},
       pages={89\ndash 102},
}

\bib{Agranovsky-Kuchment-Quinto}{article}{
      author={Agranovsky, Mark},
      author={Kuchment, Peter},
      author={Quinto, Eric~Todd},
       title={Range descriptions for the spherical mean {R}adon transform},
        date={2007},
        ISSN={0022-1236},
     journal={J. Funct. Anal.},
      volume={248},
      number={2},
       pages={344\ndash 386},
         url={https://doi.org/10.1016/j.jfa.2007.03.022},
      review={\MR{2335579}},
}

\bib{AN}{article}{
      author={Agranovsky, Mark},
      author={Nguyen, Linh~V.},
       title={Range conditions for a spherical mean transform and global extendibility of solutions of the {D}arboux equation},
        date={2010},
        ISSN={0021-7670},
     journal={J. Anal. Math.},
      volume={112},
       pages={351\ndash 367},
         url={https://doi.org/10.1007/s11854-010-0033-0},
      review={\MR{2763005}},
}

\bib{agrawal2023simple}{article}{
      author={Agrawal, Divyansh},
      author={Ambartsoumian, Gaik},
      author={Krishnan, Venkateswaran~P.},
      author={Singhal, Nisha},
       title={A simple range characterization for spherical mean transform in odd dimensions and its applications},
        date={2023},
     journal={Preprint, arXiv:2310.20702},
}

\bib{Ambartsoumian2018}{article}{
      author={Ambartsoumian, Gaik},
      author={Gouia-Zarrad, Rim},
      author={Krishnan, Venkateswaran~P.},
      author={Roy, Souvik},
       title={Image reconstruction from radially incomplete spherical {R}adon data},
        date={2018},
        ISSN={0956-7925},
     journal={European J. Appl. Math.},
      volume={29},
      number={3},
       pages={470\ndash 493},
      review={\MR{3788452}},
}

\bib{Ambartsoumian-Zarrad-Lewis}{article}{
      author={Ambartsoumian, Gaik},
      author={Gouia-Zarrad, Rim},
      author={Lewis, Matthew~A.},
       title={Inversion of the circular {R}adon transform on an annulus},
        date={2010},
        ISSN={0266-5611,1361-6420},
     journal={Inverse Problems},
      volume={26},
      number={10},
       pages={105015, 11},
         url={https://doi.org/10.1088/0266-5611/26/10/105015},
      review={\MR{2719776}},
}

\bib{ambartsoumian2015inversion}{article}{
      author={Ambartsoumian, Gaik},
      author={Krishnan, Venkateswaran~P.},
       title={Inversion of a class of circular and elliptical {R}adon transforms},
        date={2015},
     journal={Contemporary Mathematics},
      volume={653},
       pages={1\ndash 12},
}

\bib{ref:AmbKuch-range}{article}{
      author={Ambartsoumian, Gaik},
      author={Kuchment, Peter},
       title={{A range description for the planar circular Radon transform}},
        date={2006},
     journal={SIAM J. Math. Anal.},
      volume={38},
      number={2},
       pages={681\ndash 692},
}

\bib{ambartsoumian2014exterior}{article}{
      author={Ambartsoumian, Gaik},
      author={Kunyansky, Leonid},
       title={Exterior/interior problem for the circular means transform with applications to intravascular imaging},
        date={2014},
     journal={Inverse Problems and Imaging},
      volume={8},
      number={2},
       pages={339\ndash 359},
}

\bib{ambartsoumian2007thermoacoustic}{article}{
      author={Ambartsoumian, Gaik},
      author={Patch, Sarah~K.},
       title={Thermoacoustic tomography: numerical results},
        date={2007},
     journal={Proceedings of SPIE, vol. 6437, Progress in Biomedical Optics and Imaging},
      volume={8},
      number={14},
       pages={346\ndash 355},
}

\bib{AER}{article}{
      author={Antipov, Yuri~A.},
      author={Estrada, Ricardo},
      author={Rubin, Boris},
       title={Method of analytic continuation for the inverse spherical mean transform in constant curvature spaces},
        date={2012},
        ISSN={0021-7670,1565-8538},
     journal={J. Anal. Math.},
      volume={118},
      number={2},
       pages={623\ndash 656},
         url={https://doi.org/10.1007/s11854-012-0046-y},
      review={\MR{3000693}},
}

\bib{aramyan2020recovering}{article}{
      author={Aramyan, Rafik~H.},
      author={Mnatsakanov, Robert~M.},
       title={To recovering the moments from the spherical mean {R}adon transform},
        date={2020},
     journal={Journal of Mathematical Analysis and Applications},
      volume={490},
      number={2},
       pages={124334},
}

\bib{Finch-Haltmeir-Rakesh_even-inversion}{article}{
      author={Finch, David},
      author={Haltmeier, Markus},
      author={Rakesh},
       title={Inversion of spherical means and the wave equation in even dimensions},
        date={2007},
        ISSN={0036-1399,1095-712X},
     journal={SIAM J. Appl. Math.},
      volume={68},
      number={2},
       pages={392\ndash 412},
         url={https://doi.org/10.1137/070682137},
      review={\MR{2366991}},
}

\bib{Finch-P-R}{article}{
      author={Finch, David},
      author={Patch, Sarah},
      author={Rakesh},
       title={Determining a function from its mean values over a family of spheres},
        date={2004},
     journal={SIAM J. Math. Anal.},
      volume={35},
       pages={1213\ndash 1240},
}

\bib{finch2006range}{article}{
      author={Finch, David},
      author={Rakesh},
       title={The range of the spherical mean value operator for functions supported in a ball},
        date={2006},
     journal={Inverse Problems},
      volume={22},
      number={3},
       pages={923},
}

\bib{Krantz-Parks_primer}{book}{
      author={Krantz, Steven~G.},
      author={Parks, Harold~R.},
       title={A primer of real analytic functions},
     edition={Second},
      series={Birkh\"{a}user Advanced Texts: Basler Lehrb\"{u}cher. [Birkh\"{a}user Advanced Texts: Basel Textbooks]},
   publisher={Birkh\"{a}user Boston, Inc., Boston, MA},
        date={2002},
        ISBN={0-8176-4264-1},
         url={https://doi.org/10.1007/978-0-8176-8134-0},
      review={\MR{1916029}},
}

\bib{K}{article}{
      author={Kunyansky, Leonid~A.},
       title={Explicit inversion formulae for the spherical mean {R}adon transform},
        date={2007},
        ISSN={0266-5611},
     journal={Inverse Problems},
      volume={23},
      number={1},
       pages={373\ndash 383},
         url={https://doi.org/10.1088/0266-5611/23/1/021},
      review={\MR{2302980}},
}

\bib{Natterer_book}{book}{
      author={Natterer, F.},
       title={The mathematics of computerized tomography},
      series={Classics in Applied Mathematics},
   publisher={Society for Industrial and Applied Mathematics (SIAM), Philadelphia, PA},
        date={2001},
      volume={32},
        ISBN={0-89871-493-1},
         url={https://doi.org/10.1137/1.9780898719284},
        note={Reprint of the 1986 original},
      review={\MR{1847845}},
}

\bib{nguyen2009family}{article}{
      author={Nguyen, Linh~V.},
       title={A family of inversion formulas in thermoacoustic tomography},
        date={2009},
     journal={Inverse Problems and Imaging},
      volume={3},
      number={4},
       pages={649\ndash 675},
}

\bib{LVN}{article}{
      author={Nguyen, Linh~V.},
       title={Range description for a spherical mean transform on spaces of constant curvature},
        date={2016},
        ISSN={0021-7670},
     journal={J. Anal. Math.},
      volume={128},
       pages={191\ndash 214},
         url={https://doi.org/10.1007/s11854-016-0006-z},
      review={\MR{3481173}},
}

\bib{norton1980reconstruction}{article}{
      author={Norton, Stephen~J.},
       title={Reconstruction of a two-dimensional reflecting medium over a circular domain: Exact solution},
        date={1980},
     journal={The Journal of the Acoustical Society of America},
      volume={67},
      number={4},
       pages={1266\ndash 1273},
}

\bib{norton1981ultrasonic}{article}{
      author={Norton, Stephen~J.},
      author={Linzer, Melvin},
       title={Ultrasonic reflectivity imaging in three dimensions: exact inverse scattering solutions for plane, cylindrical, and spherical apertures},
        date={1981},
     journal={IEEE Transactions on Biomedical Engineering},
      volume={28},
      number={2},
       pages={202\ndash 220},
}

\bib{R}{article}{
      author={Rubin, Boris},
       title={Inversion formulae for the spherical mean in odd dimensions and the {E}uler-{P}oisson-{D}arboux equation},
        date={2008},
        ISSN={0266-5611},
     journal={Inverse Problems},
      volume={24},
      number={2},
       pages={025021, 10},
         url={https://doi.org/10.1088/0266-5611/24/2/025021},
      review={\MR{2408558}},
}

\bib{Rubin-book}{book}{
      author={Rubin, Boris},
       title={Introduction to {R}adon transforms: With elements of fractional calculus and harmonic analysis},
      series={Encyclopedia of Mathematics and its Applications},
   publisher={Cambridge University Press, New York},
        date={2015},
      volume={160},
        ISBN={978-0-521-85459-7},
      review={\MR{3410931}},
}

\bib{Salman_Article}{article}{
      author={Salman, Yehonatan},
       title={Recovering functions from the spherical mean transform with limited radii data by expansion into spherical harmonics},
        date={2018},
        ISSN={0022-247X,1096-0813},
     journal={J. Math. Anal. Appl.},
      volume={465},
      number={1},
       pages={331\ndash 347},
         url={https://doi.org/10.1016/j.jmaa.2018.05.019},
      review={\MR{3806707}},
}

\bib{xu2002time}{article}{
      author={Xu, Minghua},
      author={Wang, Lihong},
       title={Time-domain reconstruction for thermoacoustic tomography in a spherical geometry},
        date={2002},
     journal={IEEE Transactions on Medical Imaging},
      volume={21},
      number={7},
       pages={814\ndash 822},
}

\end{biblist}
\end{bibdiv}
\end{document}